\documentclass{article}
\usepackage{latexsym}
\usepackage{epsfig}
\usepackage{amsmath}
\usepackage{amssymb}
\usepackage{amsthm}
\usepackage{caption}
\newtheorem{theorem}{Theorem}[section]

\newtheorem{lemma}[theorem]{Lemma}
\newtheorem{conjecture}[theorem]{Conjecture}

\theoremstyle{remark}

\theoremstyle{definition}

\begin{document}

\title{Necessary conditions for the existence of exponential-polynomial expansions for solutions of certain differential equations}

\author{Rahaf Habib (rahoufhabib@gmail.com) \\ Roland Hildebrand  (khildebrand.r@mipt.ru) \\ MIPT, FPMI, 141701 Dolgoprudny}

\maketitle

\begin{abstract}
We consider ordinary differential equations (ODE) of the form $u''u - (u')^2 = e^{-x}P(u) - 1$, where $P$ is a polynomial. For $P = u^k$, $k = 3,4,6$ this ODE is equivalent to certain degenerate Painlev\'e III equations. We study whether families of solutions of these ODEs have asymptotic expansions of the form $u(x) = \sum_{k=0}^{\infty} p_k(x+c)e^{-kx}$ for $Re\,x \to +\infty$, where $c \in \mathbb C$ is an arbitrary constant parameterizing the solution family, $p_k$ are polynomials, with $p_0(x) = x$. We find necessary conditions on $P$ for such expansions to exist. Numerical experiments suggest that these conditions are also sufficient, and the expansions are not only formal, but actually provide a series representation of the solutions. Numerical evidence also suggests a conjecture on the nonnegativity of coefficients of the $p_k$.
\end{abstract}


\section{Introduction}

Painlev\'e transcendents appear in many areas of mathematics and mathematical physics, e.g., in the geometry of surfaces \cite{BobenkoAitner00book}. Often one is interested in specific families of solutions of these second-order ordinary differential equations (ODEs), in particular, those with pole-free regions. In this work we study solutions of an ODE which is obtained as a generalization of the degenerate Painlev\'e III equation. Namely, several instances of this equation can be transformed into members of a more general family of ODEs given by
\begin{equation} \label{generalODE}
    u''u - (u')^2 = e^{-x}P(u) - 1,
\end{equation}
where $P$ is a polynomial. The specific solutions $y(t)$ of the Painlev\'e equation transform to solutions of \eqref{generalODE} with asymptotics
\begin{equation} \label{roughAsymptotics}
    u(x) \sim x + c, \qquad x \to +\infty,
\end{equation}
where $c \in \mathbb C$ is an arbitrary constant parametrizing the family of solutions. The limit $x \to +\infty$ corresponds to the regular singular point $t = 0$ of the Painlev\'e equation.

In this work we are concerned with the following question: Can the asymptotics \eqref{roughAsymptotics} be refined into a formal series? Inserting $u(x) = x + c$ into \eqref{generalODE} leads to an exponentially small residue, suggesting that the solution is of the form $u(x) = x + c + e^{-x}p_1(x+c)$, with $p_1$ a power series. For generic polynomials $P$, substitution into the ODE yields a divergent formal power series. For certain polynomials $P$, however, including those corresponding to the transformed Painlev\'e equations, the power series is finite, hence actually a polynomial. Inserting this polynomial, we get another correction term, this time of the form $e^{-2x}p_2(x+c)$, where $p_2$ is another power series. This series can again be finite or infinite. In the former case, a correction term of the form $e^{-3x}p_3(x+c)$ can be fitted and so on.

If we do not end up with an infinite power series at some step, the formal series representing the solution will be of the form
\begin{equation} \label{fullAsymptotics}
    u(x) = \sum_{k=0}^{\infty} p_k(x+c)e^{-kx},
\end{equation}
where $c$ is the constant from \eqref{roughAsymptotics} and $p_k$ are polynomials, with $p_0(x) = x$. In Theorem \ref{thm:main} we provide necessary conditions on $P$ for a representation of the form \eqref{fullAsymptotics} to exist. 

It could be expected that at each level $k$ new conditions on the polynomial $P$ arise to ensure finiteness of the power series $p_k$. Computations and numerical experiments suggest, however, that for $k \geq 4$ the conditions found on $P$ at the previous levels are already sufficient to ensure that $p_k$ is a polynomial. If true, this entails that for polynomials $P$ from a subspace of co-dimension 2 in the space of all polynomials solutions with exponential-polynomial expansions of the form \eqref{fullAsymptotics} exist. This subspace contains the polynomials $P$ for which ODE \eqref{generalODE} is equivalent to a Painlev\'e equation.

In \cite[Conjecture C.1]{KitaevVartanianArxiv25} it has been already suggested that such series exist for special families of solutions of other degenerate Painlev\'e III equations. In the general case the $p_k$ will have infinitely many summands \cite{Kitaev25}. Convergence of exponential-polynomial series has been studied, e.g., in the papers \cite{Krivosheyeva13,Krivosheevy22}. 
The ensemble of solutions (the space of initial conditions) of the degenerate Painlev\'e III equations has been studied in detail in \cite{OKSO06}. There, in particular, types $D_7$ and $D_8$ of this ODE are distinguished. The asymptotic behaviour of the solutions has been studied in a series of papers \cite{Kitaev87,KitaevVartanian10,Kitaev19,KitaevVartanian24,Kitaev25,Vartanian25}.

The remainder of the paper is structured as follows. In Section \ref{sec:Painleve} we show how certain Painlev\'e III equations can be transformed into an ODE of the form \eqref{generalODE}. In Section \ref{sec:symmetries} we consider the symmetries of ODE \eqref{generalODE}. For generic $P$ there are only relatively simple transformations, but for the polynomials $P$ corresponding to the Painlev\'e transcendents more complicated symmetries exist. In Section \ref{sec:specialSolutions} we obtain solutions of \eqref{generalODE} for the quadratic polynomial $P(t) = t^2 + t$ together with the corresponding series \eqref{fullAsymptotics}, which seems to be the only case with an explicit solution. In Section \ref{sec:main} we deduce necessary conditions on $P$ for expansions of the form \eqref{fullAsymptotics} to exist, namely that $P(0) = 0$ and the coefficients at $t,t^2$ in $P(t)$ are equal. We also compute the degree and the leading coefficients of the polynomials $p_k$ in Lemma \ref{lem:leadingCoefficient}. In Section \ref{sec:experiments} we compute the first polynomials $p_k$ for the transformed Painlev\'e transcendents and comment on their pole-free regions. In Section \ref{sec:outlook} we formulate some conjectures which are suggested by the numerical experiments, in particular, on the nonnegativity of the coefficients of the $p_k$. Since this conjecture applies to the case of the Painlev\'e transcendents, this would imply corresponding positivity properties for their expansions.

\section{Transformation of Painlev\'e equations} \label{sec:Painleve}

Following \cite{OKSO06}, we present the degenerate Painlev\'e III equations and deduce their equivalence to ODE \eqref{generalODE}. The Painlev\'e equations are non-linear second order ODEs which contain several complex parameters. By multiplicative scaling of the independent and the dependent variable one may, however, reduce some or all of the parameters to canonical values.

\subsection{Painlev\'e III type $D_7$} \label{sec:D7transform}

The degenerate Painlev\'e III equation of type $D_7$ has the form
\[ ty\ddot y = t\dot y^2 - y\dot y + \alpha y^3 + \beta y + \delta t,
\]
with $\alpha\delta \not= 0$. Scaling acts on the parameters by the transformation $(\alpha,\beta,\delta) \gets (ab\alpha,ab^{-1}\beta,a^2b^{-2}\delta)$, where $ab \not= 0$. The quantity $\delta^{-1}\beta^2$ is an invariant of this transformation. Equations with different invariants are not equivalent.

The differential equation
\begin{equation} \label{ODE3}
    u''u - (u')^2 = e^{-x}u^3 - 1
\end{equation} 
can be transformed into the Painlev\'e III equation of type $D_7$ with $\delta^{-1}\beta^2 = 0$. By the transformations $u = e^{x/2}y$, $t = e^{-x/2}$ we obtain
\begin{equation} \label{D7equation}
ty\ddot y = t\dot y^2 - y\dot y + 4y^3 - 4t,    
\end{equation} 
which is the degenerate Painlev\'e III equation with parameters $(\alpha,\beta,\delta) = (4,0,-4)$. It has the algebraic solution $y = t^{1/3}$, which corresponds to the solution $u = e^{x/3}$ of \eqref{ODE3}. Solutions of \eqref{ODE3} with asymptotics $u \sim x + c$ as $x \to +\infty$ correspond to Painlev\'e transcendents with asymptotics $y \sim t \cdot (-2\log\,t + c)$ as $t \to +0$.


The differential equation
\begin{equation} \label{ODE6}
    u''u - (u')^2 = e^{-x}u^6 - 1
\end{equation} 
can also be transformed into equation \eqref{D7equation}. The corresponding  substitution is given by $u = \sqrt{\frac{8}{ty}}$, $x = -4\log\,t+9\log\,2$. The algebraic solution $y = t^{1/3}$ corresponds to the solution $u = e^{x/6}$ of \eqref{ODE6}. Solutions of \eqref{ODE6} with asymptotics $u \sim x + c$ as $x \to +\infty$ correspond to solutions of \eqref{D7equation} with asymptotics $y \sim 8t^{-1}\cdot (-4\log\,t+9\log\,2+c)^{-2}$ as $t \to +0$.

\subsection{Painlev\'e III type $D_8$}

The degenerate Painlev\'e III equation of type $D_8$ has the form
\begin{equation} \label{D8equation}
    ty\ddot y = t\dot y^2 - y\dot y + \alpha y^3 + \beta y
\end{equation} 
with $\alpha\beta \not= 0$. As for the $D_7$ case, scaling acts by $(\alpha,\beta) \gets (ab\alpha,ab^{-1}\beta)$. It may transform $(\alpha,\beta)$ to any given pair of non-zero numbers. The equation possesses the constant solutions $y \equiv \pm\sqrt{-\beta/\alpha}$.

The differential equation
\begin{equation} \label{symmetricEquation}
    u''u - (u')^2 = e^{-x}u^4 - 1,
\end{equation} 
which is \eqref{generalODE} with $P(t) = t^4$, is by the transformation $y = u^2e^{-x/2}$, $t = e^{-x/2}$ equivalent to the Painlev\'e III equation of type $D_8$ with parameters $(\alpha,\beta) = (8,-8)$. The constant solutions $y \equiv \pm 1$ correspond to the solutions $u = \epsilon_4^k \cdot e^{x/4}$, where $\epsilon_4 = i$ is the 4-th root of 1 and $k \in \mathbb Z$. Solutions with asymptotics $u \sim x + c$ as $x \to +\infty$ correspond to Painlev\'e transcendents with asymptotics $y \sim t \cdot (-2\log t+c)^2$ as $t \to +0$. 

\section{Symmetries of the ODE} \label{sec:symmetries}

Solutions of ODE \eqref{generalODE} may by a change of the dependent and independent variables be transformed into other solutions of ODE \eqref{generalODE} with the same or with another polynomial $P$. Below we consider a more trivial transformation valid for general $P$, and a transformation which applies only to the Painlev\'e transcendents considered in the previous section.

\subsection{Linear change of variables}

Consider the change of variables 
\[ \tilde x = x + \delta, \qquad \tilde u = \epsilon u,
\]
where $\delta \in \mathbb C$, $\epsilon = \pm1$. Then if $u(x)$ is a solution of ODE \eqref{generalODE}, we have that $\tilde u(\tilde x)$ is a solution of the ODE
\[ \frac{d^2\tilde u}{d\tilde x^2}\tilde u - \left( \frac{d\tilde u}{d\tilde x} \right)^2 = e^{-\tilde x+\delta}P(\epsilon\tilde u) - 1.
\]
This ODE is also of the form \eqref{generalODE} with a polynomial $Q(t) = e^{\delta}P(\epsilon t)$ instead of $P$. If $\delta,\epsilon$ are such that $P = Q$, then the transformation defines a symmetry of the ODE. In particular, if $u(x)$ is a solution, then $\tilde u(x) = u(x + 2\pi ik)$ is also a solution for every $k \in \mathbb Z$.

If $\epsilon = 1$, then a solution $u(x)$ with asymptotics $u \sim x + c$ is transformed into a solution $\tilde u(\tilde x)$ with asymptotics $\tilde u \sim \tilde x + (c - \delta)$, i.e., it is still in the family of interest. We may hence normalize the polynomial $P$ without loss of generality by a multiplicative constant, e.g., make it monic (i.e., leading coefficient equals 1). 

If $\epsilon = -1$, then the asymptotics of the solution is different. However, all results deduced in this work for solutions with asymptotics $u \sim x + c$ hold also in a correspondingly modified form for solutions with asymptotics $u \sim -x+c$ by means of this symmetry. If $P$ is even, then a solution $u(x)$ gives rise to another solution $\tilde u(x) = -u(x)$. If $P$ is odd, then a solution $u(x)$ gives rise to another solution $\tilde u(x) = -u(x + i\pi)$. 

\subsection{Non-linear change of variables}

Consider the change of variables
\[ \tilde x = s^{-1}x + b, \qquad \tilde u = ae^{x/2}u^{-s},
\]
where $s,a,b$ are fixed, $as \not= 0$, and let $P(t) = t^k$. Then $u = a^{1/s}e^{x/(2s)}\tilde u^{-1/s}$, $x = s(\tilde x - b)$, and
\begin{align*}
    \frac{d^2\tilde u}{d\tilde x^2}\tilde u - \left( \frac{d\tilde u}{d\tilde x} \right)^2 &= -a^2s^3e^x(u''u - (u')^2)u^{-2(s + 1)} = -a^2s^3e^x(e^{-x}u^k - 1)u^{-2(s + 1)} \\ &= a^2s^3 \left( e^x(a^{1/s}e^{x/(2s)}\tilde u^{-1/s})^{-2(s + 1)} - u^{k-2(s + 1)} \right) \\ &= a^2s^3 \left( a^{-2(1 + 1/s)}e^{-(\tilde x - b)}\tilde u^{2(1 + 1/s)} - u^{k-2(s + 1)} \right).
\end{align*} 
Choosing $a = s^{-3/2}$, $k = 2(s+1)$ we obtain
\[ \frac{d^2\tilde u}{d\tilde x^2}\tilde u - \left( \frac{d\tilde u}{d\tilde x} \right)^2 = a^{-2(1 + 1/s)}e^be^{-\tilde x}\tilde u^{2(1 + 1/s)} - 1.
\]
Choose further $b = 2(1 + s^{-1})\log\,a$, then
\[ \frac{d^2\tilde u}{d\tilde x^2}\tilde u - \left( \frac{d\tilde u}{d\tilde x} \right)^2 = e^{-\tilde x}Q(\tilde u) - 1
\]
where $Q(t) = t^l$ with $l = 2(1 + s^{-1})$.

There are three possibilities for the choice of $s$ such that $k,l$ are both positive integers, namely $s = \frac12,1,2$. Let us consider these cases separately.

\medskip

\textit{Case $s = 1$:} Then
\[ \tilde x = x, \qquad \tilde u = \frac{e^{x/2}}{u},\qquad P(t) = Q(t) = t^4.
\]
This case corresponds to a well-known symmetry of the Painlev\'e III type $D_8$ equation. If $y(t)$ is a solution of \eqref{D8equation} with $(\alpha,\beta) = (8,-8)$, then 
\[ \tilde y(t) = y(t)^{-1}
\]
is also a solution of this equation.

\textit{Case $s = \frac12$:} Then
\[ \tilde x = 2x + 9\log\,2, \qquad \tilde u = 2^{3/2}e^{x/2}u^{-1/2}, \qquad P(t) = t^3, \qquad Q(t) = t^6.
\]
Both polynomials $P,Q$ correspond to the same Painlev\'e equation of type $D_7$. If $u(x)$ has asymptotics $u \sim x + c$ for $x \to +\infty$, then $\tilde u$ has asymptotics $\tilde u \sim 2^{-1/4}e^{\tilde x/4}(\tilde x - 9\log\,2 + 2c)^{-1/2}$. Hence the solutions of \eqref{generalODE} with $P(t) = t^3$ and asymptotics $u \sim x + c$ represent a family of Painlev\'e transcendents which is different from the one represented by the solutions of \eqref{generalODE} with $P(t) = t^6$ and asymptotics $u \sim x + c$. One may hence obtain asymptotic exponential-polynomial series for these two families.

\textit{Case $s = 2$:} Then $P(t) = t^6$, $Q(t) = t^3$, and this transformation is the inverse of the transformation considered in the previous case.

\section{Explicit solutions} \label{sec:specialSolutions}

In this section we consider some explicit solutions of \eqref{generalODE}.

As was revealed in Section \ref{sec:Painleve}, the algebraic solutions of the Painlev\'e equations \eqref{D7equation},\eqref{D8equation} correspond to solutions 
\[ u(x) = e^{x/k}
\]
for $P(t) = t^k$. Although in Section \ref{sec:Painleve} the parameter $k$ equaled $3,4,6$, this solution exists for all $k \in \mathbb N_+$.

\medskip

We now consider the case $P(t) = t^2 + t$. In this case the studied exponential-polynomial series can be computed explicitly. Let us make the ansatz
\[ u(x) = \sum_{k=0}^{\infty} \frac{1}{k!}(x + H_k + c)e^{-kx}, \qquad H_k = \sum_{j=1}^k\frac{1}{j}.
\]
The polynomials $p_k(x) = \frac{1}{k!}(x + H_k)$ are all linear in this case. 

Clearly the series converges globally, as a power series in $e^{-x}$ with super-exponentially (in $k$) decaying coefficients. We have
\begin{align*}
    u'(x) + e^{-x}u &= \sum_{k=0}^{\infty} \frac{1}{k!}(1 - kx - kH_k - kc)e^{-kx} + \sum_{k=0}^{\infty} \frac{1}{k!}(x + H_k + c)e^{-(k+1)x} \\ &= 1 + \sum_{k=1}^{\infty} \frac{1}{(k-1)!}(k^{-1} - x - H_k - c)e^{-kx} + \sum_{k=1}^{\infty} \frac{1}{(k-1)!}(x + H_{k-1} + c)e^{-kx} = 1.
\end{align*} 
In this case \eqref{generalODE} can be obtained from the 1-st order linear ODE $u' + e^{-x}u = 1$. Indeed, we further get 
\[ u'' = e^{-x}u - e^{-x}u' = e^{-x}u + e^{-2x}u - e^{-x},
\]
\[ u''u - (u')^2 = e^{-x}(u^2 + u) - 1,
\]
and ODE \eqref{generalODE} with $P(t) = t^2 + t$ is satisfied. 

The solution can be more compactly written as
\[ u(x) = e^{\exp(-x)} \left( \gamma + c + \int_{e^{-x}}^{\infty} \frac{e^{-t}}{t}\,dt \right),
\]
where $\gamma$ is the Euler constant. Indeed, we have
\[ u' = -e^{-x}e^{\exp(-x)} \left( \gamma + c + \int_{e^{-x}}^{\infty} \frac{e^{-t}}{t}\,dt \right) + 1 = -e^{-x}u + 1,
\]
and hence recover the linear 1-st order ODE.

Let us check the asymptotic behaviour. For large $x$ we have
\begin{align*}
    u(x) &= \left( 1 + O(e^{-x}) \right) \left( \gamma + c + \int_{e^{-x}}^1 \frac{1}{t}\,dt - \int_{e^{-x}}^1 \frac{1-e^{-t}}{t}\,dt + \int_1^{\infty} \frac{e^{-t}}{t}\,dt \right) \\
    &= \left( 1 + O(e^{-x}) \right) \left( \gamma + c + x - \int_0^1 \frac{1-e^{-t}}{t}\,dt + O(e^{-x}) + \int_1^{\infty} \frac{e^{-t}}{t}\,dt \right) \\
    &= \left( 1 + O(e^{-x}) \right) \left( c + x + O(e^{-x}) \right) = x + c + O(xe^{-x}),
\end{align*} 
so the explicit solution is indeed identical to the series.

\section{Necessary conditions} \label{sec:main}

In this section we derive conditions on the polynomial $P$ in order for an expansion of the form \eqref{fullAsymptotics} to represent a solution of ODE \eqref{generalODE}. A precise formulation is given in Theorem \ref{thm:main} below. First we prove an auxiliary lemma.

For every integer $k \geq 1$ define linear operators $A_k,B_k$ on the space of polynomials $p(x)$, acting by
\[ A_kp = p' - kp, \qquad B_kp = x \cdot (p' - kp) - 2p.
\]
Define further the product $L_k = B_kA_k$.

\begin{lemma} \label{lem:imageOperator}
    Let $r(x) = \sum_{i=0}^{\deg\,r}\rho_ix^i$ be a polynomial. The equation $L_kp = r$ has a solution $p(x)$ if and only if $2\rho_2-2k\rho_1+k^2\rho_0 = 0$. In this case the solution is unique, and $\deg\,p = \deg\,r - 1$.
\end{lemma}

\begin{proof}
The equation $L_kp = r$ has a solution if and only if $r$ is in the image of the operator $L_k$.

We have $\deg(A_kp) = \deg\,p$, $\deg(B_kp) = \deg\,p + 1$ for every non-zero polynomial $p$, hence if $p$ is a solution of $L_kp = r$, then $\deg\,r = \deg\,p + 1$.

For every $k \geq 1$ the operator $A_k$ is injective. Since it preserves the degree, it is invertible. Therefore the image of $L_k$ equals the image of $B_k$, and $L_k$ is injective if and only if $B_k$ is.

Now we find the image of the operator $B_k$. Let $q(x)$ be an arbitrary polynomial. We may in a unique way decompose it as $q = x^2 \cdot v + l$, where $l(x) = l_1x + l_0$ for some $l_0,l_1 \in \mathbb C$ and $v$ is a polynomial. We then obtain
\begin{align*}
    B_kq &= x((x^2v + l)' - k(x^2v + l)) - 2(x^2v + l) = x^3(v' - kv) + x(l_1 - k(l_1x+l_0)) - 2(l_1x+l_0) \\ &= x^3 \cdot A_kv - kl_1x^2 - kl_0x - l_1x - 2l_0.
\end{align*} 

Since $A_k$ is invertible, a polynomial $r$ is in the image of $B_k$ if and only if the remainder of $r$ modulo $x^3$ is. Equivalently, $(\rho_2,\rho_1,\rho_0) = (-kl_1,-kl_0-l_1,-2l_0)$ for some $l_1,l_0 \in \mathbb C$. This condition is in turn equivalent to the condition on the coefficients $\rho_i$ of $r$ formulated in the lemma.

For the same reason, $B_k$ is injective if and only if its restriction to the space of affine functions $l_1x+l_0$ is injective, which is evidently the case. Hence $L_k$ is injective and the solution of $L_kp = r$, if it exists, is unique.

This completes the proof.
\end{proof}

\begin{theorem} \label{thm:main}
    Let $c \in \mathbb C$ be arbitrary and let $u(x)$ be a solution of ODE \eqref{generalODE} admitting a formal series of the form \eqref{fullAsymptotics}, where $p_0(x) = x$, and $P(t) = \sum_{j=0}^{d} \pi_jt^j$ is a polynomial with $\pi_d \not= 0$. Then $\pi_0 = 0$ and $\pi_1 = \pi_2$. In particular, $d \geq 2$. The coefficients of the $p_k$ do not depend on $c$.
\end{theorem}

\begin{proof}
    Assume the conditions of the theorem. Let us plug expansion \eqref{fullAsymptotics} into ODE \eqref{generalODE}.

    On the left-hand side of \eqref{generalODE} we get
    \begin{align*}
        u''u - (u')^2 &= \sum_{k=0}^{\infty} (p''_ke^{-kx} - 2kp'_ke^{-kx} + k^2p_ke^{-kx}) \cdot \sum_{l=0}^{\infty} p_le^{-lx} - \sum_{k,l=0}^{\infty}(p_k'e^{-kx} - kp_ke^{-kx})(p'_le^{-lx} - lp_le^{-lx}) \\ &= \sum_{k,l=0}^{\infty} (p''_kp_l - 2kp'_kp_l + k^2p_kp_l - p'_kp'_l + lp'_kp_l + kp_kp'_l - klp_kp_l)e^{-(k+l)x} \\ &= \sum_{s=0}^{\infty} \sum_{k+l=s} (p''_kp_l - p'_kp'_l - 2(k-l)p'_kp_l + k(k-l)p_kp_l)e^{-sx}.
    \end{align*} 
    Here we used the symmetry between $k$ and $l$ in the sum. The polynomials $p_k$ are evaluated at $x + c$.

    On the right-hand side we get
\begin{align*}
    P(u) &= \sum_{j=0}^d \pi_j\left( \sum_{k=0}^{\infty} e^{-kx}p_k \right)^j = \sum_{j=0}^d \pi_j\sum_{k_1,\dots,k_j=0}^{\infty} e^{-(k_1+\dots+k_j)x}\prod_{i=1}^jp_{k_i} \\ &= \sum_{s=0}^{\infty} e^{-sx} \sum_{j=0}^d \pi_j\sum_{k_1+\dots+k_j=s} \prod_{i=1}^jp_{k_i} = \sum_{s=0}^{\infty} e^{-sx} \sum_{j=0}^d \pi_j \sum_{\substack{l_0+\dots+l_s=j \\ l_1+2l_2+\dots sl_s = s}} \frac{j!}{l_0!\dots l_s!}\prod_{i=0}^sp^{l_i}_i \\
    &= \sum_{s=0}^{\infty} e^{-sx} \sum_{j=0}^d \pi_j \sum_{l_0=0}^j \begin{pmatrix}
        j \\ l_0
    \end{pmatrix} (x+c)^{l_0} \sum_{\substack{l_1+\dots+l_s=j-l_0 \\ l_1+2l_2+\dots sl_s = s}} \frac{(j-l_0)!}{l_1!\dots l_s!}\prod_{i=1}^sp^{l_i}_i \\
    &= \sum_{s=0}^{\infty} e^{-sx} \sum_{j=0}^d \pi_j \sum_{l=0}^j \begin{pmatrix}
        j \\ l
    \end{pmatrix} (x+c)^{j-l} \sum_{\substack{l_1+\dots+l_s=l \\ l_1+2l_2+\dots sl_s = s}} \frac{l!}{l_1!\dots l_s!}\prod_{i=1}^sp^{l_i}_i \\
    &= \sum_{s=0}^{\infty} e^{-sx} \sum_{l=0}^d \left( \sum_{j=l}^d \pi_j j(j-1)\cdot \dots \cdot (j-l+1) (x+c)^{j-l} \right) \cdot \left( \sum_{\substack{l_1+\dots+l_s=l \\ l_1+2l_2+\dots sl_s = s}} \frac{1}{l_1!\dots l_s!}\prod_{i=1}^sp^{l_i}_i \right) \\
    &= \sum_{s=0}^{\infty} e^{-sx} \sum_{l=0}^s P^{(l)} \cdot \left( \sum_{\substack{l_1+\dots+l_s=l \\ l_1+2l_2+\dots sl_s = s}} \frac{1}{l_1!\dots l_s!}\prod_{i=1}^sp^{l_i}_i \right) \\
    &= P + e^{-x} p_1 P' + e^{-2x} \left(p_2P' + \frac{p_1^2}{2}P''\right) + e^{-3x} \left(p_3P' + p_1p_2P'' + \frac{p_1^3}{6}P'''\right) \\
    &+ e^{-4x} \left(p_4P' + \frac{p_2^2}{2}P'' + p_1p_3P'' + \frac{p_1^2p_2}{2}P''' + \frac{p_1^4}{24}P^{(IV)}\right) \\
    &+ e^{-5x} \left(p_5P' + p_2p_3P'' + p_1p_4P'' + \frac{p_1p_2^2}{2}P''' + \frac{p_1^2p_3}{2}P''' + \frac{p_1^3p_2}{6}P^{(IV)} + \frac{p_1^5}{120}P^{(V)}\right) + \dots.
\end{align*} 
Let us comment on the calculations above. In the second line, the $k_r$, $r = 1,\dots,j$, run from 0 to $s$. The $l_i$, $i = 0,\dots,s$, run from 0 to $\lfloor \frac{s}{i} \rfloor$ and count how many $k_r$ have the value $i$. For any given $(s+1)$-tuple $(l_0,\dots,l_s)$, there are $\frac{j!}{l_0!\dots l_s!}$ possibilities for the $k_r$ to distribute the available values among themselves. In the subsequent lines we exploit that $p_0(x+c) = x+c$, which allows to treat $l_0$ distinctly from the other $l_i$. In the 4-th line, we substituted $l = j - l_0$. In the 5-th line, we collect the coefficients at the powers of $x+c$. It turns out that these are exactly the coefficients of the derivatives $P^{(l)}$ of the polynomial $P$. In the 6-th line we replace summation over $l = 0,\dots,d$ by the equivalent summation over $l = 0,\dots,s$. This is possible because in the case $l > d$ the derivative $P^{(l)}$ is the zero polynomial, and in the case $l > s$ the set of $s$-tuples $(l_1,\dots,l_s)$ satisfying
\[ l = l_1 + \dots + l_s \leq l_1 + 2l_2 + \dots + sl_s = s
\]
is empty, so the summation is in both cases effectively over $l = 0,\dots,\min(d,s)$. Finally, in the last line the polynomial $P$ itself appears only in the first term because if $s > 0$, then $l_1 + 2l_2 + \dots + sl_s = s$ is possible only if at least one of the $l_i$ is non-zero, which entails $l > 0$.

Equating the expressions for the left- and right-hand side of \eqref{generalODE} and comparing the terms at the exponents $e^{-(s+1)x}$, we finally obtain
\begin{equation} \label{finalRelation}
    \sum_{k+l=s+1} (p''_kp_l - p'_kp'_l - 2(k-l)p'_kp_l + k(k-l)p_kp_l) = \sum_{l=0}^s P^{(l)} \cdot \left( \sum_{\substack{l_1+\dots+l_s=l \\ l_1+2l_2+\dots sl_s = s}} \frac{1}{l_1!\dots l_s!}\prod_{i=1}^sp^{l_i}_i \right)
\end{equation}
for every $s \geq 0$. Here both $p_k$ and its derivatives and $P$ and its derivatives are evaluated at $x+c$. However, since this is a formal identity between polynomials, it holds also if everything is evaluated at $x$. In particular, either \eqref{fullAsymptotics} is a formal asymptotic expansion of a solution of ODE \eqref{generalODE} for every $c \in \mathbb C$, or for no $c$ at all. In the first case the coefficients of $p_k$ do not depend on $c$.

For simplicity we shall further consider $p_k,P$ as functions of $x$ and drop $c$ from our consideration.

Relation \eqref{finalRelation} yields a recursive formula for $p_{s+1}$ as a function of $p_1,\dots,p_s,P$. Namely, we get by virtue of $p_0(x) = x$ that
\begin{align} \label{p_kRecursion}
    &x(p''_{s+1} - 2(s+1)p'_{s+1} + (s+1)^2p_{s+1}) - 2(p'_{s+1} - (s+1)p_{s+1}) = L_{s+1}p_{s+1} \\ &= \sum_{l=0}^s P^{(l)} \cdot \left( \sum_{\substack{l_1+\dots+l_s=l \\ l_1+2l_2+\dots sl_s = s}} \frac{1}{l_1!\dots l_s!}\prod_{i=1}^sp^{l_i}_i \right) - \sum_{\substack{k,l \geq 1 \\ k+l=s+1}} (p''_kp_l - p'_kp'_l - 2(k-l)p'_kp_l + k(k-l)p_kp_l). \nonumber
\end{align}

For $s = 0$ we get that the polynomial $p_1$ satisfies the relation 
\begin{equation} \label{Pfrom_p1}
L_1p_1 = P.    
\end{equation} 
By Lemma \ref{lem:imageOperator} a solution $p_1$ exists if and only if 
\begin{equation} \label{relation1}
    2\pi_2-2\pi_1+\pi_0 = 0.
\end{equation}

For convenience we denote the coefficients of $p_k$ by $c_{k,j}$, such that $p_k(x) = \sum_j c_{k,j}x^j$. Let us also define a linear functional $\Lambda_k$ on the space of polynomials $p(x) = \sum_j c_jx^j$ which acts by
\[ \Lambda_kp = 2c_2 - 2kc_1 + k^2c_0.
\]
Then by Lemma \ref{lem:imageOperator} the equation $L_kp = r$ has a solution if and only if $\Lambda_kr = 0$.

The next equation \eqref{p_kRecursion} (for $s = 1$) yields
\begin{equation} \label{p2p1relation}
    L_2p_2 = p_1P' - p_1''p_1 + (p_1')^2 = p_1^2 - 2p_1''p_1 + (p_1')^2 + xp_1(p_1''' - 2p_1'' + p_1'),
\end{equation} 
where we used that by virtue of \eqref{Pfrom_p1}
\[ P' = \left( x(p_1'' - 2p_1' + p_1) - 2(p_1' - p_1) \right)' = p_1 - p_1'' + x(p_1''' - 2p_1'' + p_1').
\]
We get the condition
\[ \Lambda_2(p_1^2 - 2p_1''p_1 + (p_1')^2 + xp_1(p_1''' - 2p_1'' + p_1')) = 0.
\]
After some calculus this amounts to 
\begin{equation} \label{prerelation2}
    4(c_{1,0} - c_{1,1})(c_{1,0} - 2c_{1,1} + 2c_{1,2}) = 0.
\end{equation}

The coefficients $\pi_j$ of $P$ depend on the coefficients $c_{1,j}$ of $p_1$ by virtue of \eqref{Pfrom_p1}, which amounts to
\begin{equation} \label{piFrom_p1}
    \pi_0 = 2c_{1,0} - 2c_{1,1}, \quad \pi_1 = c_{1,0} - 2c_{1,2}, \quad \pi_2 = c_{1,1} - 2c_{1,2}, \quad \pi_3 = c_{1,2} - 4c_{1,3} + 4c_{1,4}, \dots 
\end{equation}
Relation \eqref{prerelation2} then yields $\pi_0 = 0$ in the case $c_{1,0} - c_{1,1} = 0$ and $\pi_0 - \pi_1 = 0$ in the case $c_{1,0} - 2c_{1,1} + 2c_{1,2} = 0$. This finally yields the next necessary condition
\begin{equation} \label{relation2}
    \pi_0(\pi_0-\pi_1) = 0.
\end{equation}

Let us go one step further. For $s = 2$ recursion \eqref{p_kRecursion} yields
\begin{align*}
    L_3p_3 =\,& p_2P' + \frac{p_1^2}{2}P'' - (p_1p_2 + 2p_2p_1' + p_2p_1'' - 2p_1p_2' + p_1p_2'' - 2p_1'p_2') \\
=\,& x \cdot \left( p_2p_1' - 2p_2p_1'' + p_2p_1''' + \frac{p_1^2p_1''}{2} - p_1^2p_1''' + \frac{p_1^2p_1''''}{2} \right) + \\ &+ 2p_1p_2' - 2p_2p_1'' - 2p_2p_1' - p_1p_2'' + 2p_1'p_2' + p_1^2p_1' - p_1^2p_1''.
\end{align*} 
The necessary condition 
\[ \Lambda_3\left( x\left(p_2p_1' - 2p_2p_1'' + p_2p_1''' + \frac{p_1^2p_1''}{2} - p_1^2p_1''' + \frac{p_1^2p_1''''}{2}\right) + 2p_1p_2' - 2p_2p_1'' - 2p_2p_1' - p_1p_2'' + 2p_1'p_2' + p_1^2p_1' - p_1^2p_1'' \right) = 0
\]
simplifies after some calculations to
\begin{align*}
   & -6c_{1,0}(4c_{2,2} - 8c_{2,3} + 4c_{2,4}) - (2c_{1,2} - 4c_{1,1} - 9c_{1,0})(2c_{2,1} - 2c_{2,2}) + (4c_{1,2} - 6c_{1,1})(4c_{2,0} - 2c_{2,1}) + \\ &+ 9c_{1,0}^2c_{1,1} - 12c_{1,0}c_{1,1}^2 - 8c_{1,0}c_{1,2}^2 - 36c_{1,0}^2c_{1,2} + 84c_{1,0}^2c_{1,3} - 4c_{1,1}^2c_{1,2} - 144c_{1,0}^2c_{1,4} + 120c_{1,0}^2c_{1,5} + \\ &+ 2c_{1,1}^3 + 40c_{1,0}c_{1,1}c_{1,2} - 48c_{1,0}c_{1,1}c_{1,3} + 48c_{1,0}c_{1,1}c_{1,4} = 0.
\end{align*} 
In order to eliminate the coefficients $c_{2,j}$ of $p_2$, we may use relation \eqref{p2p1relation}, which yields
\begin{align*}
    4c_{2,2} - 8c_{2,3} + 4c_{2,4} &= 5c_{1,0}c_{1,3} + 5c_{1,1}c_{1,2} - 24c_{1,0}c_{1,4} - 12c_{1,1}c_{1,3} + 20c_{1,0}c_{1,5} + 8c_{1,1}c_{1,4} + 2c_{1,2}c_{1,3} - 4c_{1,2}^2, \\
2c_{2,1} - 2c_{2,2} &= c_{1,1}^2 - 2c_{1,2}c_{1,1} + 2c_{1,0}c_{1,2} - 6c_{1,0}c_{1,3}, \\
4c_{2,0} - 2c_{2,1} &= c_{1,0}^2 - 4c_{1,2}c_{1,0} + c_{1,1}^2.
\end{align*}
By substitution the necessary condition further simplifies to
\[ (c_{1,0} - c_{1,1})(3c_{1,0}c_{1,1} - 14c_{1,0}c_{1,2} + 10c_{1,1}c_{1,2} - 4c_{1,2}^2) = 0.
\]
Adding $\frac14(3c_{1,1} - 6c_{1,0} - 2c_{1,2})$ times the left-hand side of \eqref{prerelation2}, we finally obtain the necessary condition
\[ 6(c_{1,0} - c_{1,1})^3 = 0.
\]
The first relation in \eqref{piFrom_p1} then yields $\pi_0 = 0$, and entails also \eqref{relation2}. From \eqref{relation1} we then get $\pi_1 = \pi_2$, which proves the conditions on $P$ stated in the theorem.
\end{proof}

\begin{lemma} \label{lem:leadingCoefficient}
    Assume the conditions of the previous theorem. Then for all $k$ the degree of $p_k$ equals $d_k = k(d-2) + 1$, with leading coefficient given by $c_{k,d_k} = \frac{\pi_d^k \cdot \prod_{l=0}^{k-1}(1 + \frac{(d-2)l}{2})}{k!}$.
\end{lemma}

\begin{proof}
First we show that $\deg\,p_k \leq k(d-2)+1$.

From \eqref{Pfrom_p1} we get by virtue of Lemma \ref{lem:imageOperator} that $\deg\,p_1 = \deg\,P - 1 = d - 1$. For $k \geq 2$ we prove the claim by induction. 

Suppose $\deg\,p_k \leq k(d-2) + 1$ for $k \leq s$. Let us compute the degree of the polynomial on the right-hand side of \eqref{p_kRecursion}. For non-zero derivative $P^{(l)}$ we have
\[ \deg\,\left( P^{(l)} \cdot \prod_{i=1}^sp_i^{l_i} \right) \leq d-l + \sum_{i=1}^s l_i(i(d-2)+1) = d-l + (d-2)\sum_{i=1}^s il_i + \sum_{i=1}^s l_i = d-l + (d-2)s + l = 2+(d-2)(s+1).
\]
Here we used that $\sum_i l_i = l$, $\sum_i il_i = s$. Further, using $k + l = s + 1$ we get
\[ \deg\,(p''_kp_l - p'_kp'_l - 2(k-l)p'_kp_l + k(k-l)p_kp_l) = \deg\,p_k + \deg\,p_l \leq (k+l)(d-2)+2 = (s+1)(d-2)+2.
\]

Hence the degree of the right-hand side of \eqref{p_kRecursion} does not exceed $(s+1)(d-2)+2$. But then by virtue of Lemma \ref{lem:imageOperator} the degree of $p_{s+1}$ does not exceed $(s+1)(d-2)+1$, which proves our claim.

Now we shall compute the coefficient $\rho_k = c_{k,d_k}$ of $p_k$. From the preceding considerations it follows that the recursion \eqref{p_kRecursion} boils down to a recursion on the $\rho_k$, i.e., coefficients at smaller powers than $d_k$ in $p_k$ cannot influence $\rho_{s+1}$. 

More precisely, inserting the formal expansion \eqref{fullAsymptotics} into ODE \eqref{generalODE} and neglecting all terms $x^ke^{-sx}$ with $k < d_s$, we get
\[ \sum_{k,l=0}^{\infty} (k^2-kl)\rho_k\rho_lx^{d_k+d_l}e^{-(k+l)x} = e^{-x} \cdot \pi_d \cdot \left( \sum_{k=0}^{\infty} \rho_kx^{d_k}e^{-kx} \right)^d.
\]
Define $\tilde\rho_k = \frac{\rho_k}{\pi_d^k}$, then the relation becomes
\begin{align*}
    \sum_{k,l=0}^{\infty} (k^2-kl)\tilde\rho_k\tilde\rho_l\pi_d^{k+l}x^{(d-2)(k+l)+2}e^{-(k+l)x} &= e^{-x} \cdot \pi_d \cdot \left( \sum_{k=0}^{\infty} \tilde\rho_k\pi_d^kx^{(d-2)k+1}e^{-kx} \right)^d, \\
    x^2\sum_{k,l=0}^{\infty} (k^2-kl)\tilde\rho_k\tilde\rho_l\pi_d^{k+l}x^{(d-2)(k+l)}e^{-(k+l)x} &= e^{-x} \cdot \pi_d \cdot x^d \cdot \left( \sum_{k=0}^{\infty} \tilde\rho_k\pi_d^kx^{(d-2)k}e^{-kx} \right)^d, \\
    \sum_{k,l=0}^{\infty} (k^2-kl)\tilde\rho_k\tilde\rho_lz^{k+l} &= z \cdot \left( \sum_{k=0}^{\infty} \tilde\rho_kz^k \right)^d,
\end{align*} 
where $z = \pi_d \cdot x^{d-2} \cdot e^{-x}$. Recall that $\tilde\rho_0 = \rho_0 = 1$, and as a consequence $\tilde\rho_1 = 1$.

Define a formal power series $r(z) = \sum_{k=0}^{\infty} \tilde\rho_kz^k$. Then the above relation is equivalent to the ODE
\[ r(z) \cdot \left(z\frac{d}{dz}\right)^2r(z) - \left( z\frac{d}{dz}r(z) \right)^2 = zr(z)^d.
\]
Indeed, we have $z\frac{d}{dz}r(z) = \sum_{k=0}^{\infty} k\tilde\rho_kz^k$ and by analogy $\left(z\frac{d}{dz}\right)^2r(z) = \sum_{k=0}^{\infty} k^2\tilde\rho_kz^k$. The claim then easily follows.

In order to find the coefficients $\tilde\rho_k$, we have to find a solution of this ODE with $r(0) = 1 + z + O(z^2)$ and to compute its Taylor expansion. Such a solution is readily given by
\[ r(z) = \left( 1 - \frac{d-2}{2}z \right)^{-2/(d-2)} = \sum_{k=0}^{\infty} \frac{\prod_{l=0}^{k-1}\left( 1 + \frac{(d-2)l}{2} \right)}{k!}z^k.
\]
This shows that the coefficient $\rho_k$ has the value provided in the lemma.

Since $\rho_k \not= 0$ for all $k$, we have $\deg\,p_k = d_k$, as claimed.
\end{proof}

In particular, 
\begin{itemize}
    \item for $P(t) = t^3$ we get $d_k = k + 1$, $c_{k,k+1} = \frac{k+1}{2^k}$
    \item for $P(t) = t^4$ we get $d_k = 2k + 1$, $c_{k,2k+1} = 1$
    \item for $P(t) = t^6$ we get $d_k = 4k + 1$, $c_{k,4k+1} = \frac{(2k-1)!!}{k!}$
\end{itemize}

\section{Numerical experiments} \label{sec:experiments}

In this section we present the coefficients of the first few polynomials $p_k$ for $P(t) = t^d$, $d = 3,4,6$ and analyze the resulting partial sums of the expansion \eqref{fullAsymptotics}.

First we clarify how the convergence region of the series \eqref{fullAsymptotics} might look like. Numerical evidence suggests that the coefficients of the polynomials $p_k$ are bounded by an exponential in $k$. In this case the series is convergent in a region of the following shape.

\begin{lemma} \label{lem:convergenceRegion}
    Suppose that the coefficients of the polynomials $p_k(x) = \sum_{j=0}^{d_k} c_{k,j}x^j$ of degree $d_k = (d-2)k + 1$ are bounded by $|c_{k,j}| \leq e^{\alpha k}$ for some $\alpha \in \mathbb R$. Then the series \eqref{fullAsymptotics} converges in the domain
    \[ D = \{ x \in \mathbb C \mid \max(0,(d-2)\log|x+c|) < -\alpha+Re\,x \}.
    \] 
\end{lemma}

\begin{proof}
    We have 
    \[ |p_k(x + c)| \leq e^{\alpha k}\sum_{j=0}^{d_k}|x+c|^j \leq (d_k+1)e^{\alpha k}\max(1,|x+c|^{d_k}).
    \]
    It follows that
    \begin{align*}
    |p_k(x+c)e^{-kx}| &\leq ((d-2)k+2)e^{\alpha k}\max(1,|x+c|^{(d-2)k+1})e^{-k \cdot Re\,x} \\ &= ((d-2)k+2)\max(1,|x+c|)\left( e^{\alpha-Re\,x}\max(1,|x+c|^{d-2}) \right)^k.
\end{align*} 
    Hence the series converges absolutely if $x$ satisfies the condition
    \[ e^{\alpha-Re\,x}\max(1,|x+c|^{d-2}) < 1.
    \]
    The claim of the lemma now readily follows.
\end{proof}

The actual domain of convergence of the series may be larger than the domain $D$ defined in the lemma. In particular, from the numerical experiments it can be observed that the coefficient $\alpha$ of the exponential growth rate (or sometimes rate of decline) of the coefficients $c_{k,j}$ actually depends on $\frac{j}{d_k} \in [0,1]$, with the maximum attained only at a single intermediate value. 

\subsection{Case $P(t) = t^3$}

As was established in Section \ref{sec:D7transform}, the solutions $y(t)$ of the Painlev\'e III equation \eqref{D7equation} can be obtained from solutions $u(x)$ of ODE \eqref{generalODE} with $P(t) = t^3$ by $y(t) = t \cdot u(-2\log t)$. Inserting expansion \eqref{fullAsymptotics}, we obtain
\begin{equation} \label{PainlevePower3series}
    y(t) = \sum_{k=0}^{\infty} p_k(-2\log\,t + c) \cdot t^{2k+1},
\end{equation} 
where $c \in \mathbb C$ is an arbitrary complex parameter. Table \ref{tableT3} contains the coefficients of the polynomials $p_k$, $k = 0,\dots,6$. 

\begin{table}[]
    \centering
    \begin{tabular}{c|cccccccc}
    $k$ & $x^0$ & $x^1$ & $x^2$ & $x^3$ & $x^4$ & $x^5$ & $x^6$ & $x^7$ \\
    \hline
    0 & 0 & 1 \\
    1 & 2 & 2 & 1 \\
    2 & $\frac52$ & 5 & 3 & $\frac34$ \\
    3 & $\frac{355}{81}$ & $\frac{247}{27}$ & $\frac{139}{18}$ & 3 & $\frac12$ \\
    4 & $\frac{1023025}{165888}$ & $\frac{649777}{41472}$ & $\frac{112283}{6912}$ & $\frac{20281}{2304}$ & $\frac52$ & $\frac{5}{16}$ \\
    5 & $\frac{443414657}{51840000}$ & $\frac{85778219}{3456000}$ & $\frac{1777981}{57600}$ & $\frac{8095}{384}$ & $\frac{1077}{128}$ & $\frac{15}{8}$ & $\frac{3}{16}$ \\
    6 & $\frac{31861250831}{2799360000}$ & $\frac{17521449581}{466560000}$ & $\frac{4499779717}{82944000}$ & $\frac{185225569}{4147200}$ & $\frac{627815}{27648}$ & $\frac{66097}{9216}$ & $\frac{21}{16}$ & $\frac{7}{64}$
    \end{tabular}
    \caption{Coefficients of the polynomials $p_k$ in the expansion \eqref{fullAsymptotics} for $P(t) = t^3$.}
    \label{tableT3}
\end{table}

If the coefficients of the polynomials $p_k$ are bounded by $e^{\alpha k}$, then Lemma \ref{lem:convergenceRegion} is applicable and the series \eqref{PainlevePower3series} converges in a domain of the form
\[
D = \{ t \in \mathbb C \setminus \{0\} \mid \max(1,|-2\log t+c|)\cdot |t|^2 < e^{-\alpha} \}.
\] 
If the phase of $t$ is bounded, then the domain contains a neighbourhood of the regular singular point $t = 0$.

\subsection{Case $P(t) = t^4$}

The solutions $y(t)$ of the Painlev\'e III equation \eqref{D8equation} can be obtained from solutions $u(x)$ of ODE \eqref{generalODE} with $P(t) = t^4$ by $y(t) = t \cdot u(-2\log t)^2$. Inserting expansion \eqref{fullAsymptotics}, we obtain
\begin{equation} \label{PainlevePower4series}
    y(t) = t \cdot \left( \sum_{k=0}^{\infty} p_k(-2\log\,t + c) \cdot t^{2k} \right)^2 = \sum_{s=0}^{\infty} \left( \sum_{k+l=s} p_k(-2\log\,t + c)p_l(-2\log\,t + c) \right) \cdot t^{2s+1},
\end{equation} 
where $c \in \mathbb C$ is an arbitrary complex parameter. Table \ref{tableT4} contains the coefficients of the polynomials $p_k$, $k = 0,\dots,4$.

\begin{table}[]
    \centering
    \begin{tabular}{c|cccccccccc}
    $k$ & $x^0$ & $x^1$ & $x^2$ & $x^3$ & $x^4$ & $x^5$ & $x^6$ & $x^7$ & $x^8$ & $x^9$ \\
    \hline
    0 & 0 & 1 \\
    1 & 8 & 8 & 4 & 1 \\
    2 & $\frac{145}{4}$ & $\frac{145}{2}$ & $\frac{129}{2}$ & $\frac{121}{4}$ & 8 & 1 \\
    3 & $\frac{55526}{243}$ & $\frac{50342}{81}$ & $\frac{81001}{108}$ & $\frac{9539}{18}$ & $\frac{477}{2}$ & $\frac{275}{4}$ & 12 & 1 \\
    4 & $\frac{736622003}{497664}$ & $\frac{614694323}{124416}$ & $\frac{52563371}{6912}$ & $\frac{49077907}{6912}$ & $\frac{79939}{18}$ & $\frac{31001}{16}$ & $\frac{1185}{2}$ & $\frac{493}{4}$ & 16 & 1 
    \end{tabular}
    \caption{Coefficients of the polynomials $p_k$ in the expansion \eqref{fullAsymptotics} for $P(t) = t^4$.}
    \label{tableT4}
\end{table}

Again, if $|c_{k,j}| \leq e^{\alpha k}$, Lemma \ref{lem:convergenceRegion} is applicable and the series \eqref{PainlevePower4series} converges in a domain of the form 
\[
D = \{ t \in \mathbb C \setminus \{0\} \mid \max(1,|-2\log t+c|)\cdot |t| < e^{-\alpha/2} \}.
\]

\subsection{Case $P(t) = t^6$}

The solutions $y(t)$ of the Painlev\'e III equation \eqref{D7equation} can be obtained from solutions $u(x)$ of ODE \eqref{generalODE} with $P(t) = t^6$ by $y(t) = 8t^{-1} \cdot u(-4\log t+9\log\,2)^{-2}$. Inserting expansion \eqref{fullAsymptotics}, we obtain
\begin{equation} \label{PainlevePower6series}
    y(t) = \frac{8}{t} \left( \sum_{k=0}^{\infty} p_k(-4\log t + 9\log\,2 + c) \cdot \frac{t^{4k}}{2^{9k}} \right)^{-2},
\end{equation} 
where $c \in \mathbb C$ is an arbitrary complex parameter. Table \ref{tableT6} contains the coefficients of the polynomials $p_k$, $k = 0,\dots,3$. 

\begin{table}[]
    \centering
    \begin{tabular}{c|ccccccccc}
    $k$ & $x^0$ & $x^1$ & $x^2$ & $x^3$ & $x^4$ & $x^5$ & $x^6$ & $x^7$ & $x^8$ \\
    \hline
    0 & 0 & 1 \\
    1 & 240 & 240 & 120 & 38 & 8 & 1 \\
    2 & 28233 & 56466 & 55026 & 34449 & 15312 & 4980 & 1192 & 206 & 24 \\
    3 & $\frac{3332576240}{729}$ & $\frac{3248595440}{243}$ & $\frac{1550215282}{81}$ & $\frac{1439221124}{81}$ & $\frac{321662771}{27}$ & $\frac{54792028}{9}$ & 2442162 & 780701 & 199656 \\
    \hline
    \hline
    k & $x^9$ & $x^{10}$ & $x^{11}$ & $x^{12}$ & $x^{13}$ \\
    \hline
    2 & $\frac32$ \\
    3 & 40548 & 6404 & 753 & 60 & $\frac52$
    \end{tabular}
    \caption{Coefficients of the polynomials $p_k$ in the expansion \eqref{fullAsymptotics} for $P(t) = t^6$.}
    \label{tableT6}
\end{table}

If the bound $|c_{k,j}| \leq e^{\alpha k}$ on the coefficients of the $p_k$ is valid, then by virtue of Lemma \ref{lem:convergenceRegion} the series converges in a domain
\[ \{ t \in \mathbb C \mid \max(1,|-4\log t + 9\log\,2 + c|) \cdot |t| < 2^{9/4} \cdot e^{-\alpha/4} \}.
\]

\section{Conjectures} \label{sec:outlook}

In the proof of Theorem \ref{thm:main}, we have established necessary conditions for equation \eqref{p_kRecursion} to have a polynomial solution $p_{s+1}$ for $s = 0,1,2$. Similar calculations show that at least for $s = 3,4$ no further necessary conditions appear. Numerical experiments show that the $p_s$ are indeed polynomials and can be computed accurately up to $s \approx 20$ for $d = 3$ and $s \approx 10$ for $d = 6$, at which point numerical instabilities become significant and the coefficients become imprecise. Symbolic computations for concrete examples of polynomials $P$ are able to show that system \eqref{p_kRecursion} remains solvable up to $s \approx 10$. This suggests the following conjecture.

\begin{conjecture} \label{conj1}
Let $P(t) = \sum_{j=0}^d \pi_jt^j$ be a polynomial satisfying the conditions $\pi_0 = 0$, $\pi_1 = \pi_2$, and set $p_0(x) = x$. Then for all $k \in \mathbb N_+$ there exist polynomials $p_k$ such that the formal series \eqref{fullAsymptotics} satisfies ODE \eqref{generalODE}.
\end{conjecture}

Numerical experiments also suggest the following conjectures.

\begin{conjecture}
Assume the conditions of Conjecture \ref{conj1}, and suppose in addition that $\pi_j \geq 0$ for all $j = 0,\dots,d$. Then the coefficients of all resulting polynomials $p_k$ are nonnegative.
\end{conjecture}

In particular, numerical evidence suggests that the coefficients in the expansions for the Painlev\'e transcendents corresponding to $P(t) = t^d$ with $d = 3,4,6$ are positive. This implies, in particular, that the partial sums of expansion \eqref{fullAsymptotics} can be used as lower bounds on $u(x)$ for $x + c \geq 0$, which implies corresponding bounds on the Painlev\'e transcendents.

Actually, the set of polynomials $P$ for which the coefficients of the $p_k$ are nonnegative may be larger than just the set given by $\pi_j \geq 0$. For example, for cubic polynomials $P(t) = t^3 + \beta(t^2 + t)$, where $\beta \in \mathbb R$, for decreasing $\beta < 0$ the first coefficient $c_{k,j}$ which drops below 0 is $c_{3,0} = \frac{355}{81} + \frac{4903}{972}\beta + \frac{1313}{648}\beta^2 + \frac{11}{36}\beta^3$. This suggests that the claim of the conjecture holds for $P$ with $\beta \geq \hat\beta \approx -1.9271798898$. 

\begin{conjecture}
    Assume the conditions of Conjecture \ref{conj1}. Then the conditions of Lemma \ref{lem:convergenceRegion} are satisfied and the expansion \eqref{fullAsymptotics} is actually a converging series in the domain specified by the lemma.
\end{conjecture}

We computed partial sums $\sum_{k=0}^N p_k(x)e^{-kx}$ for $P(t) = t^d$, $d = 3,4,6$ and different cut-offs $N$. Some results are presented in Figs.~\ref{fig:D3},\ref{fig:D4},\ref{fig:D6}, which depict these entire functions as complex colour plots. The convergence regions, where the partial sums are increasingly similar for growing $N$, are located right of the sequence of zeros.

\begin{figure}
    \centering
    \includegraphics[width=0.75\linewidth]{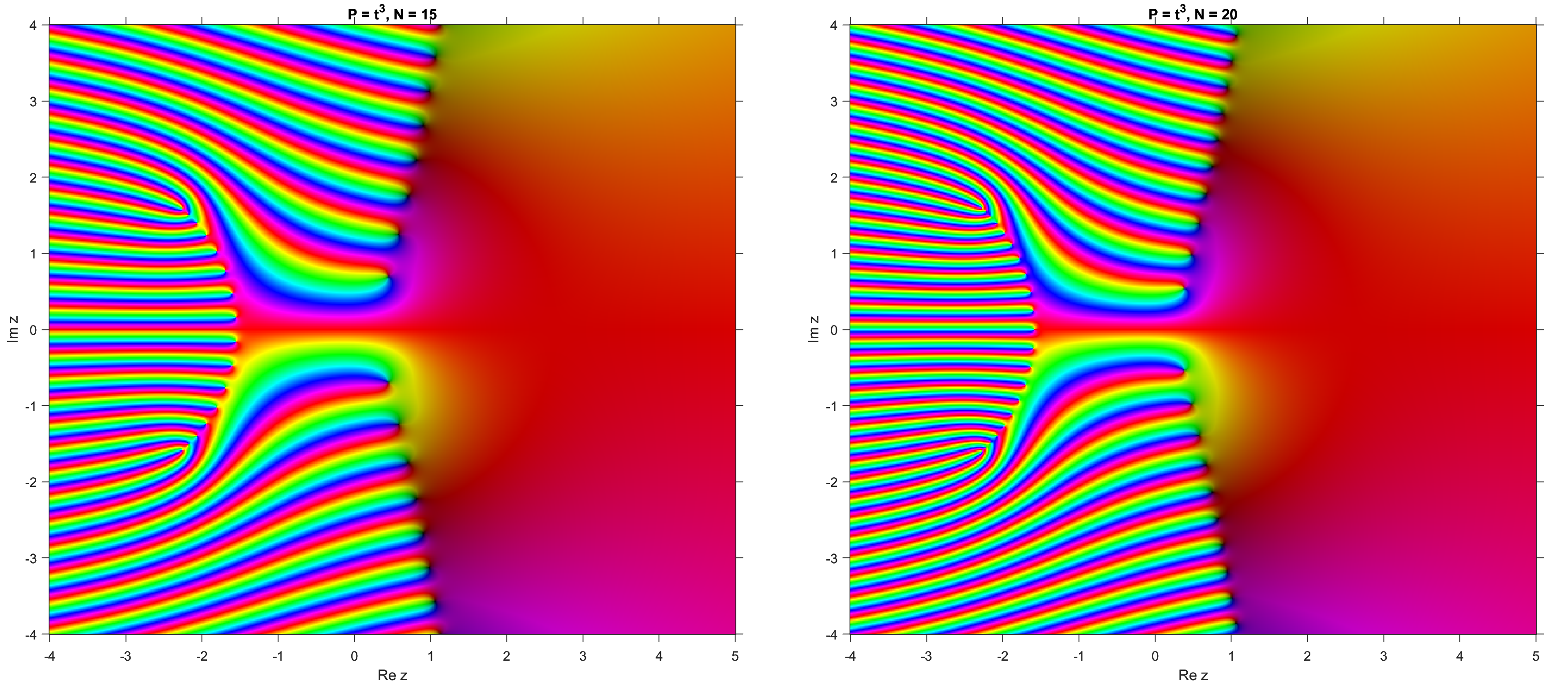}
    \caption{Partial sums of expansion \eqref{fullAsymptotics} for $P(t) = t^3$ and maximal summation index $N = 15,20$. Colour represents the phase, brightness the magnitude.}
    \label{fig:D3}
\end{figure}

\begin{figure}
    \centering
    \includegraphics[width=0.75\linewidth]{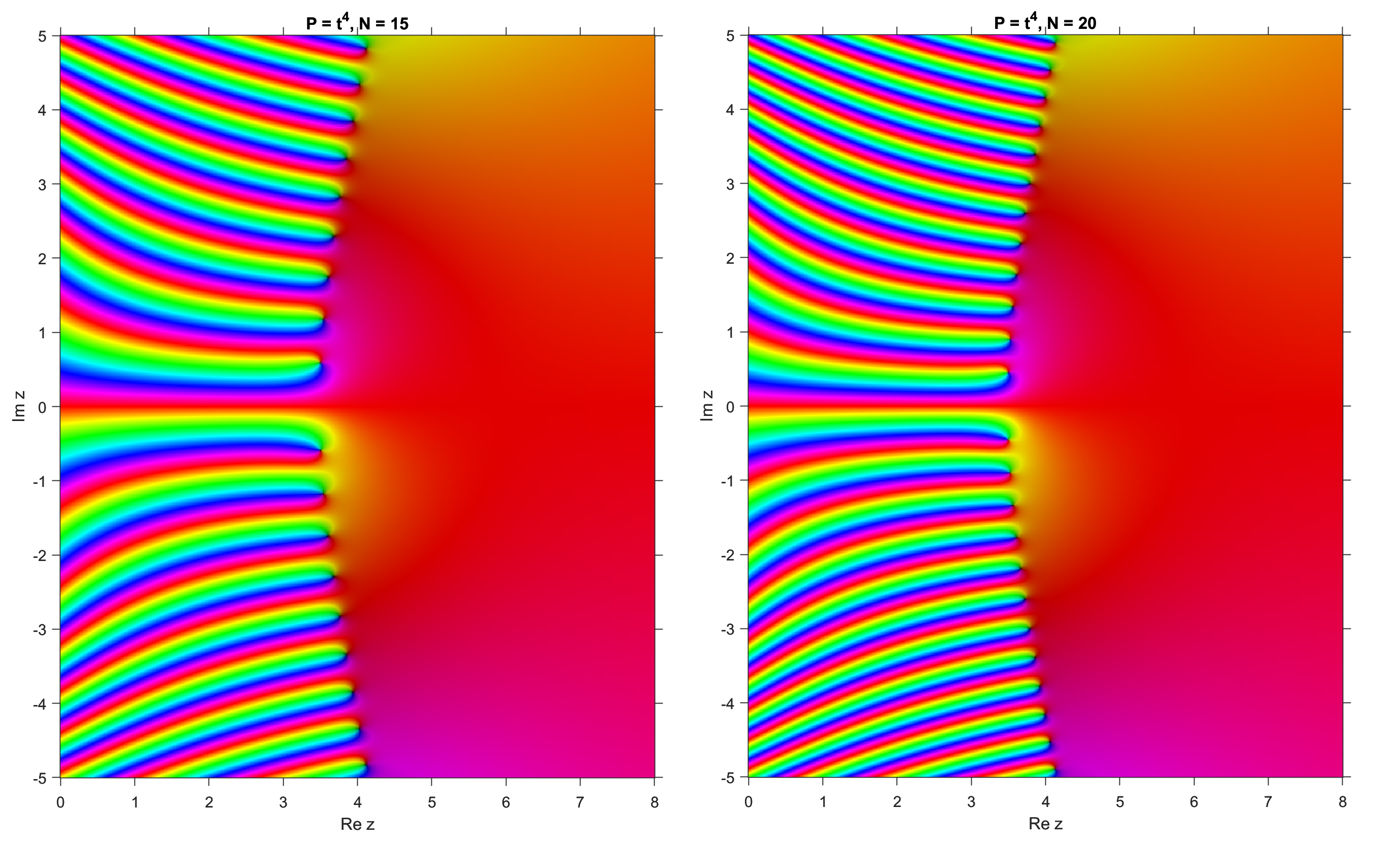}
    \caption{Partial sums of expansion \eqref{fullAsymptotics} for $P(t) = t^4$ and maximal summation index $N = 15,20$. Colour represents the phase, brightness the magnitude.}
    \label{fig:D4}
\end{figure}

\begin{figure}
    \centering
    \includegraphics[width=0.75\linewidth]{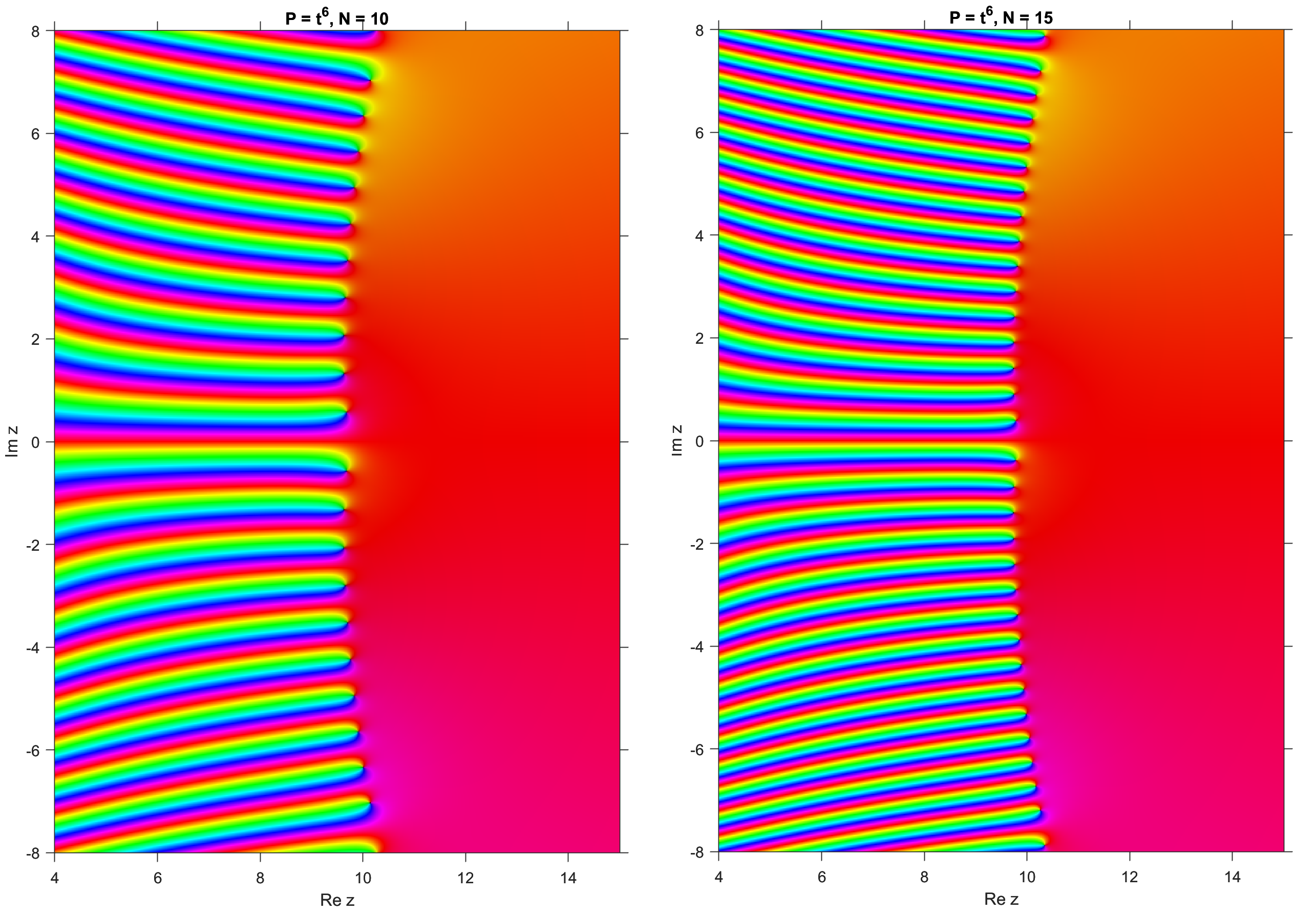}
    \caption{Partial sums of expansion \eqref{fullAsymptotics} for $P(t) = t^6$ and maximal summation index $N = 10,15$. Colour represents the phase, brightness the magnitude.}
    \label{fig:D6}
\end{figure}

\bibliographystyle{plain}
\bibliography{exponential_polynomial,Painleve}

\end{document}